\documentclass[draft,11pt]{amsart}
\usepackage{amssymb}

\theoremstyle{plain}

\theoremstyle{plain}
\newtheorem{thm}{Theorem}[section]

\newtheorem{defn}[thm]{Definition}

\newtheorem{remark}[thm]{Remark}

\newcommand{\e}{\varepsilon}
\newcommand{\T}{\mathcal T}
\newcommand{\B}{\mathcal B}

\newcommand{\R}{\mathcal R}

\newcommand{\Span}{{\rm Span}\,}
\newcommand{\N}{\mathbb{N}}

\newcommand{\Q}{\mathbb{Q}}

\def\proj{\mbox{proj}}
\def\e{\varepsilon}
\def\N{\mathbb N}
\def\B{\mathcal B}

\def\L{\mathcal L}
\def\A{\mathcal A}
\def\K{\mathcal K}
\def\S{\mathcal S}
\def\B{\mathcal B}
\def\Y{\mathcal Y}
\def\SS{\mathcal{SS}}
\def\SC{\mathcal{SC}}

\begin{document}
\baselineskip 18pt

\title[Descriptive Set Theoretic Methods]{Descriptive Set Theoretic Methods Applied to
Strictly Singular and Strictly Cosingular Operators}

\author[George~Androulakis]{George~Androulakis}
\address{Department of Mathematics,
          University of South Carolina,
          Columbia, SC 29208, USA}
\email{giorgis@math.sc.edu}

\author[Kevin~Beanland]{Kevin~Beanland}
\address{Department of Mathematics,
    Virginia Commonwealth University,
    Richmond, VA 23220, USA}
\email{kbeanland@vcu.edu}



\date{}

\begin{abstract}
The class of strictly singular operators originating from the dual of a separable Banach space is written as an increasing union of 
$\omega_1$ subclasses which are defined using the Schreier sets. A question of J.~Diestel, of whether a similar result can 
be stated for strictly cosingular operators, is studied.
\end{abstract}

\maketitle

\section{Introduction} \label{section1}

The main results of this paper use descriptive set theory to demonstrate that a class of operators
can be written as a union of $\omega_1$ many subclasses or as an intersection of $\omega_1$ many
superclasses. The Schreier classes $(S_\xi)_{\xi < \omega_1}$ which were introduced by D.~Alspach
and S.A.~Argyros (\cite{AA}) play an important role in defining these subclasses or superclasses.
Other instances where certain notions where quantified by using the Schreier families $(S_\xi)_{\xi < \omega_1}$,
are the $S_\xi$-unconditional basic sequences (\cite{F}, \cite{T}), $S_\xi$-convex combinations of sequences in Banach spaces
(\cite[page 1054]{AGR}), $S_\xi$-spreading models (\cite[page 1057]{AGR}), $\xi$-convergent sequences in Banach spaces (\cite[page 1054]{AGR}),
$\xi$-Dunford-Pettis property (\cite[page 1059]{AGR}), $S_\xi$ strictly singular operators and
$S_\xi$-hereditarily indecomposable Banach spaces (\cite{ADST}).

There are two main parts in this article. The first is an extension of a result on strictly singular operators by the first named author,
P.~Dodos, G.~Sirotkin and V.~Troitsky, (\cite{ADST}).
The classes of $S_\xi$ strictly singular operators are introduced in \cite{ADST} for $1 \leq \xi < \omega_1$. These are increasing
subclasses of the class of strictly singular operators between two fixed Banach spaces. It is proved in \cite{ADST} that
the class of strictly singular operators between two fixed separable Banach spaces is equal to the union
(for $1 \leq \xi < \omega_1$) of the classes of $S_\xi$ strictly singular operators between these spaces (Theorem~\ref{Thm:ADST}). 
The first main result of the present article states that the same holds if we merely assume that the domain space has a separable predual
(Theorem~\ref{Thm:SS}).

The second main part of this article (section~\ref{section3}) gives a partial answer to a question of J.~Diestel \cite{D}.
The question asks whether the class of strictly cosingular operators between two fixed Banach spaces can be quantified by using the Schreier 
families $(S_\xi)_{\xi < \omega_1}$
and whether the new defined classes can be used to retrieve the class of strictly cosingular operators. In other words, the question
asks whether there exists a result for strictly cosingular operators which is similar to Theorems~\ref{Thm:ADST} and \ref{Thm:SS}.
We answer this question under additional assumptions on the range space  (Theorem~\ref{Thm:Kevin}). The assumptions on the range space  
can be reduced if we work with a new class of operators whose definition is similar to the definition of strictly cosingular operators:
Given a Banach space $Y$ and a family $({\mathcal A}_i)_{i \in I}$ of normalized basic sequences of the dual space $Y^*$, we define the
notion of $({\mathcal A}_i)_{i \in I}$ strictly cosingular operators from a Banach space $X$ to $Y$ 
(their class is denoted by $(\A_i)_{i \in I}-\SC(X,Y)$).
These are defined in a similar manner to strictly cosingular operators, and for some choices of $({\mathcal A}_i)_{i \in I}$, they
coincide with the class of strictly cosingular operators. Also for $1\leq \xi < \omega_1$ we define the class of
$({\mathcal A}_i)_{i \in I}$-$S_\xi$ strictly cosingular operators from $X$ to $Y$ (their class is denoted by $(\A_i)_{i \in I}$-$\SC_\xi (X,Y)$).
These are decreasing classes of operators containing the class $(\A_i)_{i \in I}$-$\SC(X,Y)$. We prove that if $Y$ is separable 
and $\A_i$ is analytic for all $i \in I$ then the class
$(\A_i)_{i \in I}$-$\SC(X,Y)$ is equal to the intersection of $(\A_i)_{i \in I}$-$\SC_\xi (X,Y)$ for $1 \leq \xi < \omega_1$ (Theorem~\ref{Thm:SC}).

\section{Strictly singular operators}

We start by recalling some standard terminology and facts from descriptive set theory and Banach spaces. More information can be found in either \cite{K} or \cite{AGR}.
For a non-empty set $X$ let $X^{< \N}$ denote the set containing all finite sequences of elements of $X$. If $X$ is a topological
space then $X^{<\N}$ is a topological space as a direct sum of the spaces $X^n$ for all $n \in \N$.
If $X$ is a set then a tree on $X$ is a subset $\T$ of $X^{<\N}$ such that if $(x_1, x_2, \ldots , x_m) \in \T$ (for some
$m \in \N$ and $x_1, \ldots , x_m \in X$) then $(x_1, x_2, \ldots , x_n) \in \T$ for all $1 \leq n < m$. If $\T$ is a tree on
a set $X$ then an infinite branch of $\T$ is an infinite sequence $(x_n)_{n \in \N}$ of elements of $\T$ such that
$(x_1, \ldots x_n) \in \T$ for all $n \in \N$. A tree is called well founded if it does not contain infinite branches.
If $\T$ is a tree on $X$ then the $\alpha$ derivative of $\T$, $\T^{( \alpha )}$, can be defined for any $\alpha < \omega_1$ after successively applying
``$\alpha$ many trimmings of final nodes of $\T$'' (for the precise definition see either \cite[pages 1010-1011]{AGR} or \cite[page 11]{K}). 
The height of $\T$, $h(\T)$, is the least ordinal $\alpha< \omega_1$ such that $\T^{(\alpha )} = \emptyset$, if such ordinals exist, else we set $h(\T)= \omega_1$.
The following is well known (\cite[Theorem I.1.4]{AGR}).

\begin{thm} \label{pwf}
Let $X$ be a Polish space, and $\T$ be an analytic well founded tree on $X$.  Then $ h (\T) < \omega_1$.
\end{thm}

For $1 \leq \xi < \omega_1$,
the Schreier family, $S_\xi$ was introduced in \cite{AA}, (see also \cite[page 1038]{AGR}), and contains certain finite subsets of $\N$.
Let $\L(X,Y)$ denote the space of bounded linear operators from the Banach space $X$
to the Banach space $Y$. 
An operator $T \in \L(X,Y)$ is strictly singular, (denoted $T \in \SS(X,Y)$), if the restriction of $T$ to any infinite dimensional
subspace of $X$ is not an isomorphism.  In \cite{ADST}, the notion of $\S_\xi$ strictly singular operator is introduced (for $1\leq \xi < \omega_1$)
as follows. If $X,Y$ are Banach spaces, $T \in \L(X,Y)$ and $1 \leq \xi < \omega_1$, $T$ is $\S_\xi$ strictly singular
(denoted by $T \in \SS_\xi(X,Y)$) if for every $\e>0$ and every basic sequence $(x_n)$ in $X$ there exists a set $F \in \S_\xi$  and 
$x \in \Span(x_n)_{n \in F}$ such that $\|T x \| < \e \| x \|$. Thus $\SS_\xi (X,Y)$ is a subset of $\SS (X,Y)$. Also, (\cite{ADST}),
$\SS_\xi (X,Y) \subseteq \SS_\zeta (X,Y)$ if $1 \leq \xi < \zeta < \omega_1$. It is then proved in \cite{ADST} that:

\begin{thm} \label{Thm:ADST}
Let $X$, $Y$ be separable Banach spaces.  Then $\SS (X,Y) = \bigcup _{\xi < \omega_1} \SS_\xi  (X, Y)$.
\end{thm}

Here we give the following refinement of Theorem~\ref{Thm:ADST}.

\begin{thm} \label{Thm:SS}
Let $Y$ and $Z$ be Banach spaces such that $Y$ is separable and let $Y^*$ denote the dual of $Y$. 
Then $\SS (Y^* , Z) = \bigcup_{\xi < \omega_1} \SS_\xi (Y^*,Z)$.
\end{thm}

\begin{proof}
Let $Y$ and $Z$ be as in the statement and assume that $T \in \bigcap_{\xi < \omega_1} (\SS_\xi(Y^*,Z))^c$.  Let $\B$ denote the set of normalized basic
sequences in $Y^*$, $Ba(Y^*)$ denote the unit ball of $Y^*$ and $S(Y^*)$ denote the unit sphere of $Y^*$ (i.e. the vectors of norm equal to $1$). 
Let $\Y = (Ba(Y^*), \mbox{ weak}^*\mbox{ topology})^{\N}$ and ${\mathcal S}= (S(Y^*), \mbox{ weak}^*\mbox{ topology})^{\N}$.
Then $\Y$ is a Polish space, $(S(Y^*), \mbox{ weak}^*\mbox{ topology})$ is a Borel subset of $(Ba (Y^*), \mbox{ weak }^*\mbox{ topology})$
and thus ${\mathcal S}$ is a Borel subset of $\Y$.  

We claim that $\B$ is a Borel subset of ${\mathcal S}$ and thus a Borel subset of $\Y$. Indeed,
$\B = \bigcup_{k \in \N} \B_k$ where $\B_k$ denote the set of normalized basic sequences with basis constant at most
$k$. To see that $\B_k$ is a Borel set, observe that for $(y_n^*) \in {\mathcal S}$,
$$
(y_n^*)_n \in \B_k \Leftrightarrow  ~\forall N<M \text{ in }\N,~ \forall (a_n) \in \Q^{<\N},~ \forall r \in \Q^+ , 
\| \sum_{n=1}^M a_n y_n^* \| > r ~\mbox{or}~ \|\sum_{n=1}^N a_n y^*_n \| \leq kr.
$$
\noindent For fixed $N<M$ in $\N$, $(a_n)_n \in \Q^{<\N}$ and $r \in \Q^+$, the sets
$\{(y_n^*)_n \in \Y : \| \sum_{n=1}^M a_n y_n^* \| > r\} $
and  $\{(y_n^*)_n \in \Y : \|\sum_{n=1}^N a_n y^*_n \| \leq kr\} $ are open and closed, respectively in $\Y$.
Thus $\B_k$ is Borel and so $\B$ is a
Borel subset of $\S$. 

Let $\K$ be the Polish space  $\N \times \N \times \Y$.  Recall, that $\K^{<\N}$ denotes the direct sum of the 
spaces $\K^n$ (for $n \in \N$) and a tree on $\K$ will be a subset of $\K^{<\N}$.  Define a tree $\R$ on $\K$
to be the following set:
$$
\{ (m_i, \ell_i , (y^*_{i,k})_k )_{i=1}^n \in \K^{<\N} :  m_1= \cdots = m_n, (\ell_1 < \cdots < \ell_n) \in \N^{<\N}, 
(y^*_{1,k})_k=  \cdots =  (y^*_{n,k})_k \in \B \} .
$$
\noindent
Note that a generic element of $\R$ will look like $(m,\ell_i, (y^*_k)_k)_{i=1}^n$ where $m,n \in \N$, 
$(\ell_1< \ell_2< \cdots < \ell_n) \in \N^n$ and $(y_k^*)_k \in \B$.  Obviously  $\R$ 
is a closed subset of $( \N \times \N \times \B )^{< \N }$ thus it is a
Borel subset of $\K^{<\N}$, (since $(\N \times \N \times \B)^{<\N}$ is a Borel subset of $\K^{<\N}$).  
For each $m \in \N$, define the following subtree of $\R$.
$$
\T= \{ (m,\ell_i, (y^*_{k})_k )_{i=1}^n \in \R :   ~\forall (a_i)_i \in \Q^{<\N}, 
\| T  \sum_{i=1}^n a_i y_{\ell_i}\| > \frac{1}{m}  \|  \sum_{i=1}^n a_i y_{\ell_i}\|  \} . 
$$
\noindent $\T$ is a Borel subtree of $\R$, (since $T$ is $\| \cdot \| $-weak$^*$ continuous and $\| \cdot \|$ is weak$^*$ lower semicontinuous).  
Since $T \in \bigcap_{\xi < \omega_1} (\SS_\xi(Y^*,Z))^c$, for all $\xi < \omega_1$ there is 
an $m \in \N$ and $(y^*_k)_k \in \B$ such that  for all $(\ell_1, \ldots , \ell_n) \in S_\xi$ we have that  $(m,\ell_i,(y_k^*)_k)_{i=1}^n \in \T$.  
The subtree $\T_{(y_k^*)_k,m}$ of $\T$ 
containing nodes of the form $(m,\ell_i,(y_k^*)_k)_{i=1}^n$ with  $(\ell_1, \ldots , \ell_n) \in S_\xi$ is order isomorphic to $S_\xi$.
Thus for each $\xi < \omega_1$ the height of $\T$ is greater than or equal to $h(S_\xi)=\omega^\xi$.  Whence $h(\T)=\omega_1$.  By 
Theorem~\ref{pwf}, $\T$ is not well founded and thus there is an $m \in \N$, $(y^*_k)_k \in \B$ and $(\ell_i)_{i=1}^\infty$ such that for all $ n \in \N$
$(m,\ell_i,(y_k^*)_k)_{i=1}^n \in \T$.  This implies that $T|_{[y_{\ell_i}^*]_{i=1}^\infty}$ is an isomorphism and thus $T$ is not in $SS(Y^*,Z)$. 
\end{proof}

\section{Strictly cosingular operators} \label{section3}

For Banach spaces $X$ and $Y$ and an operator $T \in \L (X,Y)$,  A.~Pelczynski  defined in \cite{P1} that  $T$ is called
strictly cosingular
if for any subspace $Z$ of $Y$ of infinite codimension, the operator $Q_Z T$ is not onto, where $Q_Z$ is the
canonical quotient map from $Y$ to $Y/Z$. We denote by $\SC (X,Y)$ the set of all strictly cosingular operators from $X$ to $Y$.
Notice that by dualizing, $T \in \L (X,Y)$ is strictly cosingular if and only if for any
infinite dimensional weak$^*$-closed subspace $W$ of
$Y^*$, the restriction of $T^*$ on $W$, $T^*|_W$, is not an isomorphism.  Thus, in particular, if $T^* \in \SS (Y^*, X^*)$ then 
$T \in \SC (X,Y)$. The converse is in general false, as it can be seen from the inclusion operator from $c_0$ to $\ell_\infty$.
Pelczynski proved in \cite{P1} that this operator is strictly cosingular but its adjoint is a projection from the dual of $\ell_\infty$
to $\ell_1$ which fails to be strictly singular. For Banach spaces $X$ and $Y$ let $SS_*(X,Y)$ denote the set of all operators
$T \in \L (X,Y)$ such that $T^* \in \SS (Y^*, X^*)$. For every $\xi < \omega_1$ define  the classes 
$SS_{\xi, *}(X,Y)$ to contain all operators $T \in \L(X,Y)$ such that $T^* \in \SS_\xi (Y^*,X^*)$. Note that if $Y$ is separable
then by Theorem~\ref{Thm:SS}, $\SS_*(X,Y)= \bigcup\limits_{\xi < \omega_1}\SS_{\xi,*}(X,Y)$. Thus, if the range space $Y$ is separable 
and $\SC(X,Y)=\SS_*(X,Y)$ then we obtain an answer to J.~Diestel's question.

Recall that W.T.~Gowers answered in the negative, in \cite{G}, the question of whether every Banach space contains a boundedly complete 
basic sequence or an isomorph $c_0$. Consequently V.P.~Fonf characterized in \cite{Fo} the class ${\mathcal K}$ of Banach spaces that 
contain either a boundedly complete basic sequence or $c_0$. Let ${\mathcal K}_s$ denote the class of Banach spaces such that all of their
infinite dimensional closed subspaces contain a boundedly complete basic sequence or an isomorph of $c_0$. Many known Banach spaces
belong to ${\mathcal K}_s$. For example, since every space with an unconditional basis contains $c_0$, $\ell_1$ or a reflexive subspace, we obtain 
that ${\mathcal K}_s$ contains all Banach spaces which are saturated with unconditional basic sequences. Also by the results of 
W.B.~Johnson and H.P.~Rosenthal, in \cite[Theorem~IV.1.(ii)]{JR}, all separable dual spaces belong to ${\mathcal K}_s$.

\begin{thm} \label{Thm:Kevin}
Let $Y$ be separable and $Y^* \in {\mathcal K}_s$. Then for any Banach space $X$, $\SC(X,Y)= \SS_*(X,Y)$. Thus by Theorem~\ref{Thm:SS},
$\SC(X,Y) = \bigcup_{\xi < \omega_1} \SS_{\xi,*}(X,Y)$. 
\end{thm}

\begin{proof}
Let $T \in \SC(X,Y)$ and suppose $T^*$ is not in $\SS(Y^*,X^*)$. Find a subspace $Z$ of $Y^*$ such $T^*|_{Z}$ is
an isomorphism. Our goal is to find an infinite dimensional weak$^*$ closed subspace preserved by $T^*$. 
This will contradict the fact that $T \in \SC(X,Y)$, and establish that $\SC(X,Y)= \SS_*(X,Y)$. The fact that 
$\SC(X,Y) = \bigcup_{\xi < \omega_1} \SS_{\xi,*}(X,Y)$ follows immediately from Theorem~\ref{Thm:SS} since $Y$ is separable.

Since $Y^*$ is in $\K_s$, we first assume that $Z$ contains a normalized boundedly complete basic sequence, call it $(y_n^*)$. Since $Y$
is separable, the weak$^*$ topology on the unit ball of $Y^*$ is metrizable, thus $(y_n^*)$ has a weak$^*$ convergent subsequence. 
Let $(z_n^*)$ be the difference sequence of that weak$^*$ convergent subsequence of $(y_n^*)$. Then $(z_n^*)$ is weak$^*$ null 
seminormalized boundedly complete basic sequence in $Z$. By \cite[Theorem III.1]{JR}, since $Y$ is separable,
$(z_n^*)$ has a further subsequence $(z_{k_n}^*)$ which is weak$^*$ basic (and boundedly complete). By \cite[Proposition II.1]{JR}, the weak$^*$ and
norm closure of the linear span of $(z_{k_n}^*)$ coincide and thus $Z$ contains an infinite dimensional weak$^*$ closed subspace.
The restriction of $T^*$ on that subspace is an isomorphism. 

In the second case we assume that $Z$ contains and isomorph of $c_0$. We recall that H.P.~Rosenthal proved in \cite[Theorem 1.3]{R} that if
an operator originates from a Banach space which is complemented in its second dual and preserves an isomorphic copy of $c_0$ then it preserves
an isomorphic copy of $\ell_\infty$. Since the dual Banach space $Y^*$ is complemented in its second dual, by the above result of Rosenthal, 
there exists a subspace $W$ of $Y^*$ which is isomorphic to $\ell_\infty$ such that the restriction of $T^*$ on $W$ is an isomorphic embedding.
Obviously $W$ contains a weak$^*$  closed subspace of $Y^*$ (consider, for example, a reflexive subspace of $W$). The restriction of $T^*$
on that weak$^*$ closed subspace is also an isomorphism.
\end{proof}

The assumption on $Y^*$ in Theorem~\ref{Thm:Kevin} can be eliminated if we work with the class of $(\A _i)_{i \in I}$ strictly cosingular
operators instead of the class of strictly cosingular operators. We now motivate the definition of $(\A _i)_{i \in I}$ strictly cosingular operators
by examining the definition of strictly cosingular operators. Let $X$, $Y$ be Banach spaces and $T \in \L (X,Y)$. Since every weak$^*$ closed subspace 
of $Y^*$ contains the weak* closed linear span
of some basic sequence we have that $T \in \SC(X,Y)$ if and only if for any basic sequence
$(y_n^*)_n$ in $Y^*$, $T^*|_{\widetilde{\Span}(y^*_n)}$ is not an isomorphism
(where $\Span (y_n^*)$ denotes the linear span of $(y_n^*)$ and $\widetilde{\Span} (y_n^*)$ denotes the weak$^*$ closure
of the linear span of $(y_n^*)$).
When attempting to define subclasses of strictly cosingular operators in a way similar to the definition of $S_\xi$ strictly singular
operators, we encounter an immediate problem.
Namely, whether or not $T^*|_{\widetilde{\Span}(y^*_n)}$ is an isomorphism is not characterized by whether or not
$T^*|_{\Span(y_n^*)_{n\in F}}$ is an
isomorphism with the same isomorphism constant for any finite subset $F$ of $\N$.  A natural way to compensate for this is
to consider complements of the finite sets. Of course,  $T^*|_{\widetilde{\Span}(y^*_n)}$ is not an isomorphism
for all basic sequences $(y_n^*)$ in $Y^*$, if and only if
for all basic sequence $(y_n^*)$ in $ Y^*$
and for every finite subset  $I$ of $\N$, $T^*|_{\widetilde{\Span}(y^*_n)_{n \not\in I}}$ is not an isomorphism.
A new problem arises: If for every basic sequence $(y_n^*)$ in $Y^*$ one examines the restriction of $T^*$ on
$\widetilde{\Span}(y^*_n)_{n \not\in I}$ where $I$ belong to a fixed class $\mathcal{I}$ of finite subset of $\N$
(e.g. $\mathcal{I}=S_\xi$ for some $\xi < \omega_1$), then the fact that we examine every basic sequence
$(y_n^*)_n$, makes the class $\mathcal{I}$ unimportant, since $(y_n^*)_{n \not\in I}$ is another basic sequence in $Y^*$.  
We fix this problem by fixing a family $({\mathcal A}_i)_{i \in I}$ of sets ${\mathcal A}_i$ of basic sequences in
$Y^*$ and examining whether, for every $i \in I$ there exists a basic sequence $(y_n^*) \in {\mathcal A}_i$ such that
for all finite subsets $F$ of positive integers, $T^*|_{\widetilde{\Span}(y_n^*)_{n \not \in F}}$
is not an isomorphism. Here is the precise definition:

\begin{defn}
Let $X$, $Y$ be Banach spaces and $T \in \L (X,Y)$. Let $I$ be an index set and for every $ i \in I$ let
${\mathcal A}_i$ be a set of basic sequences in the dual space $Y^*$ of $Y$. Fix $1 \leq \xi < \omega_1$.
\begin{enumerate}
\item We say that $T$ is $({\mathcal A}_i)_{i \in I}$ strictly cosingular, (denoted by
$T \in ({\mathcal A}_i)_{i \in I}$-$\SC (X,Y)$), if and only if for every $i \in I$ there exists
$(y_n^*)_n \in {\mathcal A}_i$ and an infinite subset $N$ of the positive integers such that for every
finite subset $F$ of $N$ we have that $T^*|_{\widetilde{\Span}(y_n^*)_{n \not \in F}}$ is not an isomorphism.
\item We say that $T$ is $({\mathcal A}_i)_{i \in I}$-$S_\xi$ strictly cosingular ,(denoted by
$T \in ({\mathcal A}_i)_{i \in I}$-$\SC_\xi(X,Y)$), if and only if for every $i \in I$ there exists
$(y_n^*)_n \in {\mathcal A}_i$ and an infinite subset $N$ of the positive integers such that for every
finite subset $F$ of $N$ with $F \in S_\xi$ we have that $T^*|_{\widetilde{\Span}(y_n^*)_{n \not \in F}}$ is
not an isomorphism.
\end{enumerate}
\end{defn}

First note that for Banach spaces $X$ and $Y$ and any families $({\mathcal A}_i)_{i \in I}$ of basic sequences of $Y^*$, 
we have $\SC (X,Y) \subseteq ({\mathcal A}_i)_{i \in I}$-$\SC(X,Y)$.
Also note that for a specific choice of $({\mathcal A}_i)_{i \in I}$, we have that 
$\SC (X,Y) = ({\mathcal A}_i)_{i \in I}$-$\SC(X,Y)$.
Indeed, consider the set ${\mathcal W}$ of all weak$^*$ closed infinite
dimensional subspace of $Y^*$, and for every $W \in {\mathcal W}$,
consider the set ${\mathcal A}_W$ of all normalized basic sequences in $W$.
Then it is easy to see that  $({\mathcal A}_W)_{W \in {\mathcal W}}$-$\SC (X,Y) = \SC (X,Y)$.
Also, $({\mathcal A}_W)_{W \in {\mathcal W}}$-$\SC_\xi (X,Y) = \SC (X,Y)$ for all $\xi <\omega_1$ thus the next
Theorem~\ref{Thm:SC} does not give new results for strictly cosingular operators.

It is obvious that for Banach spaces $X$, $Y$, a family $({\mathcal A}_i)_{i \in I}$ of sets of basic sequences in
$Y^*$ and an ordinal $1 \leq \xi < \omega_1$ we have that
$$
\text{$({\mathcal A}_i)_{i \in I}$-$\SC (X,Y)$} \subseteq \text{$({\mathcal A}_i)_{i \in I}$-$\SC _\xi(X,Y)$}.
$$
Also notice that for ordinals $1\leq \xi < \zeta < \omega_1$ we have that
$$
\text{$({\mathcal A}_i)_{i \in I}$-$\SC_\zeta (X,Y) \subseteq ({\mathcal A}_i)_{i \in I}$-$\SC_\xi (X,Y)$}.
$$
For the last inclusion, let $T \in ({\mathcal A}_i)_{i \in I}$-$\SC_\zeta (X,Y)$ and for $i \in I$
let a basic sequence $(y_n^*)_n \in {\mathcal A}_i$ and an infinite subset $N$ of positive integers such that
for all $F \subset N$ with $F \in S_\zeta$, $T^*|_{\widetilde{\Span}(y_n^*)_{n \not \in F}}$ is not an isomorphism.
There exists an infinite subset $M$ of $N$ such that for every $F \subset M$ with $F \in S_\xi$ we have that
$F \in S_\zeta$ (\cite{AGR}). Thus for every $F \subset M$ with $F \in S_\xi$,
$T^*|_{\widetilde{\Span}(y_n^*)_{n \not \in F}}$ is not an isomorphism. Thus
$T \in ({\mathcal A}_i)_{i \in I}$-$\SC_\xi (X,Y)$.

The following result is the analogue of Theorems~\ref{Thm:ADST} and \ref{Thm:SS} for $({\mathcal A}_i)_{i \in I}$ strictly cosingular operators.

\begin{thm} \label{Thm:SC}
Let $X$ and $Y$ be Banach spaces such that $Y$ is separable, and $({\A _i})_{i \in I}$ be a family
an analytic sets of basic sequences in $Y^*$.  Then
$(\A_i)_{i \in I}$-$\SC (X,Y) = \bigcap_{\xi < \omega_1} ( \A_i )_{i \in I}$-$\SC_\xi (X,Y)$.
\end{thm}

\begin{proof}
It suffices to show that $\bigcap_{\xi < \omega_1} ( \A_i)_{i \in I}$-$\SC_\xi (X,Y) \subset ( \A_i )_{i \in I}$-$\SC (X,Y)$.
Fix an operator $T \in \bigcap_{\xi < \omega_1} ( \A_i )_{i \in I}$-$\SC_\xi (X,Y)$ and $i_0 \in I$.
As in the proof of Theorem~\ref{Thm:SS}, let $\B$ denote the set of normalized basic
sequences in $Y^*$, $Ba(Y^*)$ denote the unit ball of $Y^*$, $\Y$ denote the Polish space $Ba(Y^*)^\N$ with the product
of the weak$^*$ topology.
Let $\K$ be the Polish space  $\N \times \Y$.  Recall, that $\K^{<\N}$ denotes the direct sum of the 
spaces $\K^n$ (for $n \in \N$) and a tree on $\K$ will be a subset of $\K^{<\N}$.  Define a tree $\R$ on $\K$
as follows:
$$
\R = \{ ( \ell_i , (y^*_{i,k})_k )_{i=1}^n \in \K^{<\N} :  (\ell_1 < \cdots < \ell_n) \in \N^{<\N}, 
(y^*_{1,k})_k=  \cdots =  (y^*_{n,k})_k \in \B \} .
$$
\noindent
Note that a generic element of $\R$ will look like $(\ell_i, (y^*_k)_k)_{i=1}^n$ where $n \in \N$, 
$(\ell_1< \ell_2< \cdots < \ell_n) \in \N^n$ and $(y_k^*)_k \in \B$.
It has been shown in the proof of Theorem~\ref{Thm:SS} that  $\B$ is a Borel subset of $\Y$. 
Thus $(\N \times \B)^{<\N }$ is a Borel subset of $\K ^{< \N } = (\N \times \Y )^{ < \N }$. Obviously  $\R$ 
is a closed subset of $( \N \times \B )^{< \N }$ thus it is a
Borel subset of $\K^{<\N}$.  Define the following two subtrees of $\R$.
$$
\T = \{ (\ell_i, (y^*_{k})_k )_{i=1}^n \in \R :   ~ T^*|_{\widetilde{\Span}(y_k^*)_{k \not\in \{ \ell_1, \ldots , \ell_n \} }} 
\mbox{ is not an isomorphism} \} ,
$$
$$
\T_{i_0} = \{ (\ell_i, (y^*_{k})_k )_{i=1}^n \in \R : (y_k^*)_k \in \A_{i_0} \} .
$$
\noindent Both $\T$ and $\T_{i_0}$ are trees on $\K$.  The main component of the proof will be to show that
$\widetilde{\T}=\T\cap \T_{i_0}$ is an analytic subtree of $\R$, and thus an analytic tree on $\K$.

In order to see that $\T_{i_0}$ is an analytic tree, define the map $\pi : \R \to \B$ by 
$\pi ( (\ell_i, (y^*_k)_k)_{i=1}^n) = (y_k^*)_k$ where $((\ell_i, (y^*_k)_k)_{i=1}^n)$ is a generic element of $\R$. 
Obviously this map is continuous, and $\T_{i_0} = \pi^{-1} (\A_{i_0})$. Since $\A_{i_0}$ is an analytic set,
we obtain that $\T_{i_0}$ is analytic (see e.g. \cite[Proposition~$14.3$(ii)]{K}). 

To demonstrate that $\T$ is analytic,  we wish to write it as a countable
intersection of analytic sets.  Let,

\begin{equation*}
\begin{split}
\mathcal{C}_\e = \{ ((\ell_i, (y^*_{k})_k )_{i=1}^n, y^*) \in &~ \R \times (Ba(Y^*),\mbox{weak}^*\mbox{ topology}) :
y^* \in \widetilde{\Span}(y_k^*)_{k \not\in \{\ell_1, \ldots , \ell_n\}} \\
&  ~\mbox{and}~ \| T^* y^* \| < \e \|y^*\| \}.
\end{split}
\end{equation*}

\noindent  In this definition, as well as the rest of the proof, $Ba (Y^*)$ will be equipped with the weak$^*$ topology.
Clearly, $\T = \bigcap_{\e \in \Q^+} \proj_\R \mathcal{C}_\e$ where $\proj_\R$ denotes the projection to $\R$.  
Since $\R$ is a Borel subset of $\K ^{<\N}$ there exists a finer Polish topology on $\K^{<\N}$ making $\R$ a clopen set
and having the same Borel sets as the ordinary topology of $\K ^{<\N}$ (\cite[Theorem 13.1]{K}). Considering this finer
topology on $\R$ we have that $\R \times (Ba (Y), \mbox{ weak}^*\mbox{ topology})$ is Polish. Thus once we show that 
$\mathcal{C}_\e$ is Borel we obtain that $\proj_\R \mathcal{C}_\e$ is analytic (\cite[Exercise 14.3]{K}) hence $\T$ is 
analytic. Decompose $\mathcal{C}_\e$ as the intersection of the following two sets,
$$
\mathcal{C}_\e^1 = \{ ( (\ell_i,(y_k^*)_k)_{i=1}^n, y^*) \in \R \times Ba(Y^*), : \| T^* y^* \| < \e \|y^*\| \} ,
$$
$$
\mathcal{C}^2 = \{ ( (\ell_i,(y_k^*)_k)_{i=1}^n, y^*) \in \R \times Ba (Y^*) : y^* \in \widetilde{\Span}(y_k^*)_{k \not\in \{\ell_1, \ldots , \ell_n\}} \} .
$$
\noindent For $ (( \ell_i ,(y_k^*)_k)_{i=1}^n, y^*) \in \R \times Ba (Y^*)$ notice that 
$$
(( \ell_i ,(y_k^*)_k)_{i=1}^n, y^*) \in \mathcal{C}^1_\e \Leftrightarrow \forall r \in \Q^+,~ \|y^* \| > r ~\mbox{or}~ \| T^* y^* \| \leq \varepsilon r . 
$$
\noindent Hence, $\mathcal{C}_\e^1$ is a Borel subset of $\R \times Ba (Y^*)$ (by the weak$^*$ lower semicontinuity of $\| \cdot \|$).

Now we demonstrate that ${\mathcal C}^2$ is a Borel set. Let $F(Y^*)$ denote the topological space of the weak$^*$ closed subsets of 
$Y^*$. This is called ``the Effros-Borel space of the topological space $(Y^*, \mbox{weak}^*\mbox{ topology})$'', (see \cite[section~$12$.C]{K}),
and its basic open sets have the form
$V_U:= \{ F \in F(Y^*): F \cap U \not = \emptyset \}$ (where $U$ is an arbitrary weak$^*$ open subset of $Y^*$). 
Define the map, $\psi : \R \times Ba (Y^*) \rightarrow F(Y^*)$, by
$$
\psi((i,(y^*_n)_n)_{i \in I},y^*) = \{ z^* - y^* : z^* \in \widetilde{\Span}(y^*_n)_{n \not\in I} \} .
$$

\noindent We claim that $\psi$ is continuous.  Indeed, fix a weak* open subset $U$ in $Y^*$ and notice that,

\begin{equation*}
\begin{split}
\psi^{-1}(V_U) & = \{  (( i ,(y_k^*)_k)_{i \in I},y^*) \in \R \times Ba (Y^*): \{ z^*-y^*: z^* \in \widetilde{\Span}(y^*_n)_{n \not\in I} \}  \cap U \not= \emptyset \} \\
& = \{ ( (i,(y_k^*)_k)_{i \in I},y^*) \in \R \times Ba (Y^*) : \{ z^*-y^*: z^* \in \Span (y^*_n)_{n \not\in I}\} \cap U \not= \emptyset \} ,
\end{split}
\end{equation*}
(since $U$ is weak$^*$ open, 
$\Span (y^*_n)_{n \not\in I} \cap (y^*+U) \not = \emptyset \Leftrightarrow \widetilde{\Span}(y^*_n)_{n \not\in I} \cap (y^* +U) \not = \emptyset$). 
Thus 
$$
\psi^{-1}(V_U) = \bigcup\limits_{I \in \N^{<\N},~(a_i) \in {\mathbb R}^{<\N}} {\mathcal C}_{I, (a_i)}
$$
where for fixed $I \in \N^{<\N}$ and $(a_i) \in {\mathbb R}^{<\N}$,
$$
{\mathcal C}_{I, (a_i)}:= \{  ((i,(y_k^*)_k)_{i \in I},y^*) \in \R  \times Ba(Y^*):   \sum_{i \not\in I} a_i y_i^*-y^* \in U \} .
$$
Since $U$ is a weak$^*$ open subset of $Y^*$, ${\mathcal C}_{I, (a_i)}$ is an open subset of $\R  \times Ba(Y^*)$, showing that $\psi$ is continuous.
Since $Y$ is separable, $\{ 0 \} $ is a weak$^*$ closed set thus the set $\{F \in F(Y^*): 0 \in F\}$ is $G_\delta$  in $F(Y^*)$ and therefore
$$
\mathcal{C}^2 = \psi^{-1}(\{F \in F(Y^*): 0 \in F\}) = \{ ((i,(y_k^*)_k)_{i \in I}, y^*) \in \R \times B^{w^*} : y^* \in \widetilde{\Span}(y^*_n)_{n \not\in I} \}
$$
\noindent  is Borel. 

Having determined that $\mathcal{C}^2$ and $\mathcal{C}_\e^1$ are both Borel, it follows that $\mathcal{C}_\e$ is Borel for each $\e \in \Q^+$.  
Finally $\widetilde{\T}$ is an analytic tree on $\K$.  

By our assumption, for each $\xi < \omega_1$ there exist a $(y^*_n) \in \A_{i_0}$ and an infinite subset 
$N$ of $\N$ such that for any $(\ell_1< \ldots < \ell_n) \subset N$ with $(\ell_1, \ldots, \ell_n) \in S_\xi$, $ (\ell_i,(y_k^*)_k)_{i=1}^n \in \widetilde{\T}$.  
The subtree $\widetilde{\T}_{(y_n^*)_n, N}$ of $\widetilde{\T}$ containing all the nodes of the form $ (\ell_i,(y_k^*)_k)_{i=1}^n$ with 
$(\ell_1 < \ldots < \ell_n) \subset N$ and $(\ell_1, \ldots, \ell_n) \in S_\xi$ is order isomorphic to $S_\xi$.  Thus, for each $\xi$ the height of the tree 
$\widetilde{\T}$ is greater that or equal to $h(S_\xi)=\omega^\xi$.  Whence, $h(\widetilde{\T})=\omega_1$.  By Theorem \ref{pwf},  conclude that 
$\widetilde{\T}$ is not well founded.  Let $(y^*_n)_n \in \A_{i_0}$ and $(\ell_i)_{i=1}^\infty$ be an increasing sequence of 
positive integers such that for all $m \in \N$,  $ ( \ell_i ,(y_k^*)_k)_{i=1}^m \in \widetilde{\T}$.  Therefore, if $F$ is any finite subset of 
$(\ell_i)_{i=1}^\infty \subset \N$, $T^*|_{\widetilde{\Span}(y^*_{n})_{n \not\in F}}$ is not an isomorphism.  Since the above is true for all $i_0 \in I$, 
$T \in (\A _i)_{i \in I}$-$\SC (X,Y)$.
\end{proof}

We now give some examples of analytic sets of basic sequences in a dual Banach space $Y^*$ for separable $Y$. Such sets can be used for 
$\A_i$ in defining $(\A_i)_{i \in I}$-$\SC (X,Y)$. Bossard, in \cite{B}, shows that for a separable Banach space $Z$, the set of all boundedly complete and 
the set of all shrinking basic sequence is coanalytic (the complement of an analytic set) non-Borel as a subset of $Z^\N$.

\begin{remark} \label{banda}
Assume that $Y$ is a separable Banach space. Let $W$ be a fixed infinite dimensional weak$^*$ closed subspace of $Y^*$, 
$(z_n^*)$ be a fixed basic sequence in $Y^*$ and $\A$ be a fixed analytic set of basic sequences in $Y^*$. Let $\Y$ be the Polish space
$(Ba (Y^*), \mbox{ weak}^* \mbox{ topology})^{\N}$.
\begin{enumerate}
\item The set $\B_W$, of normalized basic sequences in $W$ is a Borel subset of $\Y$.
\item If $W$ has infinite codimension in $Y^*$ then the set $\B_{W^c}$ of all normalized basic sequences $(y_n^*)$ in $Y^*$ 
with $y_n^* \not \in W$ for all $n \in \N$, is a Borel subset of $\Y$.
\item The set $\B_{u}$, of all unconditional basic sequences is a Borel subset of $\Y$.
\item The set $\A^{sub}_{(z_n^*)}$ of all subsequences of $(z_n^*)$ is an analytic subset of $\Y$.
\item The set $\A^{block}_{(z_n^*)}$ of all block sequences of $(z_n^*)$ is an analytic subset of $\Y$.
\item The set $\A_{nbc}$, of all non-boundedly complete basic sequences in $Y^*$ is an  analytic subset of $\Y$.
\item The set $\A_{bo}$, of all sequences in $\A$ with biorthogonal vectors in $Y$ is an analytic subset of $\Y$.
\end{enumerate} 
\end{remark}

\begin{proof}
\noindent (1) is proved in Theorem~\ref{Thm:SS}. (2) is obvious since $Y^* \backslash W$ is a weak$^*$ open set.

\noindent (3) This follows from writing $\B_{u}$ as,
\begin{equation*}
\begin{split}
\B_{u} & = \{ (y_n^*)_n \in \B : \exists k \in \N, \forall (a_i)_i \in \Q^{<\N}~ \forall F ~\mbox{finite subset of}~ \N,  
          \| \sum_{i \in F} a_i y_i^* \| < k \| \sum_i a_i y_i^* \| \} \\
& = \bigcup_{k \in \N} \bigcap_{(a_i)_i \in \Q^{<\N}} \bigcap_{\substack{F \subset \N \\ 
finite}}  \{ (y_n^*)_n \in \B : \forall r \in \Q^+,  \| \sum_i a_i y_i^* \| > r  ~ \mbox{or}~ \| \sum_{i \in F} a_i y_i^* \| \leq kr \} .
\end{split}
\end{equation*}

\noindent (4) We have that $(w_n^*) \in \A^{sub}_{(z_n^*)}$ if and only if there exists $k_1<k_2< \cdots$ an increasing sequence 
of positive integers such that $w_n^* = z_{k_n}^*$ for all $n \in \N$. The set $[ \N ]^{\N}$ of the increasing sequences of integers endowed with the
coordinate-wise convergence is a Polish space, thus the result follows from \cite[Exercise~$14.3$ (ii)]{K}.

\noindent (5) We have that $(w_n^*) \in \B^{block}_{(z_n^*)}$ if and only if there exists $k_1<k_2< \cdots $ an increasing sequence of 
integers and $(a_i) \in {\mathbb R}^{\N}$ such that $w_n^* = \sum\limits_{i=k_n}^{k_{n+1}-1}a_i z_i^*$. Since $[ \N ]^{\N} \times {\mathbb R}^{\N}$
endowed with the coordinate-wise convergence is a Polish space, the result follows from \cite[Exercise~$14.3$ (ii)]{K}.

\noindent (6)  Set $\B = \B _{Y^*}$ (see part (1)) and 
$$
\mathcal{C}_{nbc} := \{ ((y_n^*)_n,(a_n)_n) \in \B \times \mathbb{R}^\N :  \exists M , \sup_{N} \| \sum_{n=1}^N a_n y^*_n \| \leq M \&
~ \bigg(\sum_{n=1}^N a_n y^*_n \bigg)_N ~\mbox{is not Cauchy} \} .
$$
\noindent  Observe that $\A_{nbc}=\proj_\B \mathcal{C}_{nbc}$ (the projection on $\B$).    Therefore, it is enough to show that 
$\mathcal{C}_{nbc}$ is a Borel subset of $\B \times \mathbb{R}^\N$ (since we can assume that $\B$ is a Polish space by \cite[Theorem 13.1]{K}).  
For each $M \in \N$, the set
$$
\mathcal{C}^M = \{ ((y_n^*)_n,(a_n)_n) \in \B \times \mathbb{R}^\N :  \sup_{N} \| \sum_{i=n}^N a_n y^*_n \| \leq M \} 
$$
\noindent is a closed in $\B \times \mathbb{R}^\N$.  Thus $\bigcup_{M \in \N} \mathcal{C}^M$ is Borel.  Additionally,

\begin{equation*}
\begin{split}
\mathcal{D}= \{ ((y_i^*)_i,&(a_i)_i) \in \B \times \mathbb{R}^\N :  (\sum_{i=1}^N a_i y^*_i )_N ~\mbox{is not Cauchy} \} \\
&  = \{ ((y_i^*)_i,(a_i)_i) \in \B \times \mathbb{R}^\N :  \exists r \in \Q^+, \forall n \in \N, \exists ~n < m \in \N , ~  \|\sum_{i=n}^m a_i y^*_i \| > r  \} \\
& = \bigcup_{r \in \Q^+} \bigcap_{n \in \N} \bigcup_{n<m \in \N}  \{ ((y_i^*)_i,(a_i)_i) \in \B \times \mathbb{R}^\N :  \|\sum_{i=n}^m a_i y^*_i \| > r  \}
\end{split}
\end{equation*}

\noindent is Borel. Thus, $\mathcal{C}_{nbc}=\mathcal{D} \bigcap (\bigcup_{M \in \N} \mathcal{C}^M)$ is Borel.\\

\noindent (7) Set $ \mathcal{C}_{bo}:= \{((y_n^*)_n, (y_n)_n) \in \Y \times Y^\N : y^*_k(y_n) = \delta_{k,n} \}$. Then
$\mathcal{C}_{bo}$ is a closed, and hence trivially a Borel subset of $\Y \times Y^\N$.
$\A_{bo}$ is analytic since $\A_{bo}=\proj_\Y \mathcal{C}_{bo}$ (the projection on $\Y$). \\
\end{proof}

Finally we present an example in order to illustrate Theorem~\ref{Thm:SC}. Let $D$ be the dyadic tree
$\{ \emptyset \} \cup \{0,1\}^{< \N}$ and consider the space $X= \ell_1 (D)$. Let $\{e_\alpha : \alpha \in D \}$ denote the
basis of X which is ordered as: $e_\emptyset, e_0, e_1, e_{0,0}, e_{0,1},
e_{1,0}, e_{1,1}, e_{0,0,1}$, etc. Fix $\omega_1$ many infinite branches of
$D$, and enumerate them as $\{ b_\xi: \xi < \omega_1 \}$. For fixed $\xi < \omega_1$, let $(e_{\xi, n})_n$ be the
increasing enumeration of the set $\{e_\alpha: \alpha \in b_\xi \}$. Also
for every $\xi < \omega_1$ let $F_\xi$ be a fixed finite set which does not belong in
$S_\xi$. Define $T\in \L (X)$  to be equal to $0$ on $e_{\xi , n}$ for every $\xi$ and every
$n \in F_\xi$. Also define $T$ to be the identity on the rest of the basis of
$X$. Then the biorthogonal functional of the basis become a weak$^*$ basis for $X^*$ (see \cite{JR}), 
 and $T^*$ is $0$ on $e_{\xi, n}^*$ (the biorthogonal functional of $e_{\xi , n}$) for every $\xi$ and
every $n \in F_\xi$. Also $T^*$ is identity on the rest of the weak$^*$ basis.
Now let ${\mathcal A}$ denote the set of subsequences of $\{e_\alpha^*: \alpha
\in D \}$ which are infinite branches of $D$ ($e_\alpha^*$ denotes the biorthogonal functional of $e_\alpha^*$ for all $\alpha \in D$). 
This is a closed subset $(Ba(X^*), weak^*)^{\mathbb N}$ (this is because each ``level'' of $D$ has finitely many nodes thus finite many possible limits). Since $F_\xi$ is
not in $S_\xi$ we have that $T \in {\mathcal A}$-$\SC_\xi (X,X)$ for all $\xi$. Thus Theorem~\ref{Thm:SC} gives the existence of an infinite 
branch $b$ of $D$ and an infinite subset $N$ of $b$ such that for all finite subsets $F$ of $N$ we have that 
$T^*|_{\widetilde{\Span} (e_\alpha)_{\alpha \in b \backslash F}}$ is not an isomorphism.

\end{document}